\documentclass[reqno, a4paper, draft]{amsart}
\usepackage{amssymb, amscd, amsfonts,  amsthm}

\usepackage[mathscr]{eucal}
\usepackage{url} 

\usepackage{amssymb,amsfonts,
amsmath,latexsym,bm,mathtools,comment}
\usepackage{multirow}
\usepackage{mathabx}

\usepackage[mathscr]{eucal}
\usepackage{pgf,tikz}
\usepackage{enumerate}

\usetikzlibrary{positioning,decorations,matrix,arrows}

\newtheorem{thm}{Theorem}[section]
\newtheorem{lem}[thm]{Lemma}
\newtheorem{coro}[thm]{Corollary}
\newtheorem{prop}[thm]{Proposition}

\theoremstyle{definition}

\newenvironment{newlist}
   {\begin{list}{}{\setlength{\labelsep}{0.25cm}
                   \setlength{\labelwidth}{0.65cm}
                      \setlength{\leftmargin}{0.9cm}}}
   {\end{list}}

\newcommand{\dotcup}{\dot{\cup}}

\DeclareMathOperator{\IScP}{{\ope{IS} _{\mathrm{c}}\ope{P}^+}}
\newcommand{\class}[1]{\boldsymbol{\mathcal{#1}}}
\newcommand{\cat}[1]{\boldsymbol{\mathscr{#1}}}
\newcommand{\str}[1]{\mathbf{#1}}
\newcommand{\alg}[1]{\str{#1}}
\newcommand{\spc}[1]{\str{#1}}
\newcommand{\fnt}[1]{\mathsf{#1}}
\newcommand{\ope}[1]{\mathbb{#1}}

\newcommand{\CA}{{\cat A}}
\newcommand{\CB}{\cat B}
\newcommand{\duplicateV}{\cat A}
\newcommand{\baseV}{\cat B}

\newcommand{\TL}{{\cat{DT}}}  

\newcommand{\CL}{{\cat L}}   
\newcommand{\CP}{\cat P}
\newcommand{\CY}{\cat Y}
\newcommand{\CX}{\cat X}
\newcommand{\CN}{\class N}
\newcommand{\CM}{\class M}
\newcommand{\DM}{\cat{DM}} 
\newcommand{\KL}{\cat{K}} 
\newcommand{\CCD}{\cat D}
\newcommand{\DB}{\cat{DB}}

\newcommand{\DBC}{\cat{DBC}}
\newcommand{\unbounded}[1]{#1\mbox{\tiny{$u$}}}

\newcommand{\DBU}{\unbounded{\DB}}
\renewcommand{\epsilon}{\varepsilon}  

\newcommand{\twiddle}[1]{{\smash{\underset{\raise.375ex\hbox{$\smash\sim$}}
       {\mathbf{#1}}}\vphantom{\underline{\str{#1}}}}} 

\newcommand{\A}{\alg{A}}
\newcommand{\T}{\alg{T}}
\newcommand{\B}{\alg{B}}
\newcommand{\C}{\alg{C}}
\newcommand{\M}{\alg{M}}
\newcommand{\X}{\alg{X}}
\newcommand{\N}{\alg{N}}
\newcommand{\PP}{\alg{P}}  
\newcommand{\Q}{\alg{Q}}
\newcommand{\two}{\boldsymbol 2}
\newcommand{\Lalg}{\alg{L}}
\newcommand{\four}{\boldsymbol 4}
\newcommand{\Tp}{{\mathscr{T}}}

\newcommand{\Sigmavar}[1]{\Sigma_{\scriptscriptstyle #1}}
\newcommand{\Gammavari}[1]{\Gamma_{\scriptscriptstyle #1}}
\newcommand{\fourvar}[1]{\four_ {\scriptscriptstyle #1}}
\newcommand{\Svar}[1]{{\bf 16}_ {\scriptscriptstyle #1}}

\DeclareMathOperator{\ISP}{\ope{ISP}}
 \DeclareMathOperator{\HSP}{\ope{HSP}}
\DeclareMathOperator{\aaimp}{\supset}

\newcommand{\odotG}[1]{\mathrel{{\odot}_{#1}}} 

\renewcommand{\leq}{\leqslant}
\renewcommand{\geq}{\geqslant}

\newcommand{\stwiddle}[1]{\smash{\underset{\smash{\raise.1ex\hbox{\small$\approx$}}}
                         {\str{#1}}}\vphantom{#1}}

\newcommand{\D}{\fnt D}
\newcommand{\E}{\fnt E}

\newcommand{\MT}{\twiddle{\spc{M}}}
\newcommand{\NT}{\twiddle{\spc{N}}}

\newcommand{\twoT}{\twiddle 2}

\newcommand{\Y}{\spc{Y}}

\begin{document}
\title[Natural dualities through product representations]{Natural dualities through product representations: bilattices and
beyond}

\author{L.\,M.\,Cabrer}
\address[L.\,M.\,Cabrer]{Institute of Computer Languages, 
Technische Universit\"at Wien,
Favoritenstrasse 9-11, A-1040 Wien, Austria}
\email{leonardo.cabrer@logic.at}

\author{H.\,A. Priestley}
\address[H.\,A. Priestley]{
Mathematical Institute, University of Oxford, Radcliffe Observatory Quarter,
Oxford OX2 6GG, United Kingdom}
\email{hap@maths.ox.ac.uk}

\begin{abstract}
This paper focuses on natural dualities for varieties of bilattice-based algebras.
Such varieties have been widely studied as semantic models  
in  situations where information is incomplete or inconsistent. 
The most popular tool for  studying bilattices-based algebras is product representation.
The 
authors recently set up a widely applicable algebraic framework which enabled 
product representations over a base variety  to be derived in a uniform
and categorical  manner.  
By combining this methodology with that of natural duality theory, we demonstrate how to build   a natural duality 
for any bilattice-based variety which has a  suitable product representation over a 
dualisable base variety.  This procedure allows us  systematically to present economical 
natural dualities   for many bilattice-based varieties, for most of which 
no dual representation has previously been given. 
 Among our results we highlight that for bilattices with  a generalised 
 conflation  operation (not assumed to
be an involution or  
 commute with negation).  
 Here both   the associated product  representation and the duality are new.
Finally we outline analogous procedures for pre-bilattice-based algebras (so negation is absent).
 \end{abstract}


\maketitle

\section{Introduction} \label{sec:intro}

 Bilattices, with and without additional operations, 
have been identified by researchers in artificial intelligence
 and in philosophical logic as of value for analysing  scenarios
   in which  information may be incomplete or inconsistent.
  Over twenty years, a bewildering array of different mathematical 
models has been developed
which employ  bilattice-based algebras in such situations;  
\cite{Gin,MaTr,FitAnnProg,RF} give just a sample 
of the literature.  
Within a logical context,
bilattices have been used to interpret truth values of 
 formal systems.
The range of possibilities is illustrated by    
 \cite{ND3,AA1, F97,G99, Fnice, BR11, SW11, OW15}.

To date, the  structure theory of bilattices has had two main strands:  
 product representations 
(see in particular \cite{BJR11,D13,prod} and references therein)   
and topological duality theory \cite{MPSV,JR,CPOne}.  
In this paper we entwine these two  strands, demonstrating how 
a dual representation and a product representation can be expected to fit together
and to operate in a symbiotic way.  Our work on distributive bilattices in  \cite{CPOne} provides a prototype.   
Crucially, 
as in~\cite{CPOne}, we exploit the theory of natural dualities; see Section~\ref{Sec:NatDual}.

In \cite{prod} we set up  a uniform framework for product 
representation.   
We introduced a formal definition of \emph{duplication}
of a base variety of algebras which gives rise to a new variety 
with  additional  operations built by  combining  suitable  
algebraic terms in the base language and coordinate manipulation 
(details are recalled in Section~\ref{Sec:Prelim}).  
This construction  led to a very general categorical 
theorem on product representation \cite[Theorem~3.2]{prod} 
which  makes overt the intrinsic structure of such representations.  
The examples we present below all involve bilattice-based varieties, 
but we stress that the scope of the theorem is not confined to such varieties.   
Our Duality Transfer Theorem (Theorem~\ref{thm:TransferNaturalDualities}) 
demonstrates how  
a natural duality for a  given base class immediately yields  a natural duality  
for any duplicate of that class.  
Moreover, the dualities for duplicated varieties 
mirror those for the base varieties,
 as regards both advantageous properties
and  complexity (note the concluding remarks in Section~\ref{NatViaDupl}). 
By combining  the Duality Transfer Theorem  with product representation
we can set up  dualities for assorted bilattice-based varieties 
(see Section~\ref{NatViaDupl}, Table~1).  
In almost all cases the dualities are new.
The varieties in question arise as duplicates of~$\CB$ (Boolean algebras), 
$\CCD$ (bounded distributive lattices)~$\KL$ (Kleene algebras), 
$\DM$ (De Morgan algebras), and  $\DB$ (bounded distributive bilattices),  
all of which have amenable natural dualities 
(see \cite{CD98} and also \cite{CPOne}).   
Variants are available  when lattice bounds are omitted.

We contrast  key features of our natural duality approach with  
earlier work on dualities for bilattice-based algebras.  
We stress that our methods lead directly to  dual representations 
which are categorical: morphisms do not have to be treated 
case-by-case as an overlay to an object representation 
(as is done in \cite{MPSV,JR}). 
Others'  work on dualities in the context of distributive 
bilattices has sought instead, for a chosen class of algebras, 
 a dual category  which is  an enrichment of a subcategory of Priestley spaces,
that is, they start from Priestley duality, applied to the distributive lattice reducts of their algebras,   
and then superimpose  extra structure to capture the suppressed operations. 
This strategy has  been successfully applied to very many 
 classes  of distributive-lattice-based algebras, but it has  drawbacks.  
Although  the underlying Priestley duality is natural,   the enriched Priestley
space representation rarely is.  
Accordingly one cannot expect 
the rewards a natural duality offers, such as instant access to free algebras.

Section~\ref{sec:Conflation} focuses on  
 the variety $\DB_{\boldsymbol -} $ of 
(bounded) distributive bilattices with a  conflation 
 operation~$-$ which is not assumed 
 to be an involution or to 
commute with the negation.  
 This variety  has not been investigated before
 and would not have  susceptible  to  earlier methods.  
We realise $\DB_{\boldsymbol -}$ as a duplicate of the variety $\cat{DO}$
of double Ockham algebras and set up  a natural duality for~$\cat{DO}$, 
whence we obtain a duality for~$\DB_{\boldsymbol -}$.  
Both results are new.
This example is also a novelty within bilattice theory since it 
takes us outside the realm of finitely generated varieties without 
losing the benefits of having a natural duality.

In Section~\ref{Sec:prebilat} we consider the negation-free 
setting of pre-bilattice-based algebras, and  link the ideas of \cite[Section~9]{prod} 
with dual representations.  
Again, a very general theorem enables us
to transfer a known duality from a base variety to a suitably constructed duplicate. 
 Here multisorted duality theory  is needed. 
  Nonetheless the ideas and the categorical arguments are simple, and 
the proof of Theorem~\ref{thm:TransferNaturalDualities}
is easily adapted.


\section{The general product representation theorem recalled }
\label{Sec:Prelim}

We shall assume that readers are familiar  with the basic notions concerning 
bilattices.
A summary can be found, for example, in \cite{BJR11} and a bare minimum in
 \cite[Section~2]{prod}.  
Here we simply draw attention to some salient points 
concerning notation  and  terminology since usage in the literature varies.  
Except in Section~\ref{Sec:prebilat} we  assume  
that a negation operator is present.

A \emph{\upshape{(}unbounded\upshape{)} bilattice} is an algebra
$\A = (A; \lor_t,\land_t,\lor_k,\land_k, \neg)$, where the reducts
$\A_t\coloneqq   (A; \lor_t,\land_t)$ and $\A_k \coloneqq  (A; \lor_k,\land_k)$ 
are lattices (respectively the \emph{truth lattice} and \emph{knowledge lattice}).  
The operation  $\neg$, capturing negation, is an endomorphism of $\A_k$ 
and a dual endomorphism of $\A_t$.

Bilattice models come in two flavours:  with and without bounds.  
Which flavour is preferred (or appropriate) may depend on  an intended application, 
or on mathematical considerations.  
We refer to \cite[Section~1]{CPOne} for the formal 
definition of the terms  \emph{bounded} and \emph{unbounded}.  
Here we merely issue a reminder that when universal bounds for the lattice order are not included 
in the algebraic language for a class of lattice-based algebras then the algebras
involved may, but need not, have bounds; when bounds do exist these
do not have to be preserved by homomorphisms.   
A subscript $_u$ on the symbol denoting a category will indicate 
that we are working in the unbounded setting.  
So, for example, $\CCD$ denotes the category of 
bounded distributive lattices and $\CCD_u$ the category of 
all distributive lattices.

All the bilattices considered in this paper are distributive, meaning that each 
 of the four lattice operations distributes over each of the other three. 
The weaker condition of  interlacing is 
necessary and sufficient for a bilattice to have a product representation.  
 However varieties of interlaced bilattice-based algebras
seldom come within the scope of natural duality theory.

Our investigations involve classes of algebras, viewed both  algebraically and categorically.  
We  draw, lightly,  on some of the basic formalism and theory of universal algebra, 
specifically regarding varieties ({\it alias} equational classes) and prevarieties;
 a standard reference for this material is \cite{BS}. 
 A class of algebras over a common language
 will be regarded as a category in the usual way: 
the morphisms  are all the homomorphisms. 
The variety generated by a family $\CM$ of algebras of common type
is denoted  $\mathbb V(\CM)$.
 Equivalently $\mathbb V(\CM)$ is the class  
$\HSP(\CM)$  of homomorphic images of subalgebras of products of algebras in $\CM$.  
The \emph{prevariety} generated by $\CM$ is the class
$\ISP(\CM)$ whose members 
are isomorphic images of subalgebras of products of members of~$\CM$. 
 Usually the algebras in~$\CM$ will be finite. 

We now recall our general product representation framework
\cite[Section~3]{prod}. 
We fix  an arbitrary  algebraic language $\Sigma$
and let $\CN$ be  a  family of $\Sigma$-algebras.   
Let $\Gamma$ be a set of pairs of $\Sigma$-terms 
such that, for $(t_1,t_2)\in\Gamma$, the terms 
$t_1$ and $t_2$ have common  even arity, 
denoted  $2n_{(t_1,t_2)}$.
We  view $\Gamma$ as  an algebraic language 
for a family of algebras  $\fnt{P}_{\Gamma}(\N)$
($\N \in \CN$), where the arity of $(t_1,t_2)\in\Gamma $ is $n_{(t_1,t_2)}$.  
We write $[t_1,t_2]$ when the pair 
$(t_1,t_2)$ is regarded  as belonging to~$\Gamma$, \textit{qua} language. 
 For $\A \in \mathbb{V}(\CN)$ we define a $\Gamma$-algebra 
$\fnt{P}_{\Gamma}(\A)
=(A\times A; \{ [t_1,t_2]^{\fnt{P}_{\Gamma}(\A)} \mid (t_1,t_2)\in\Gamma \})$,
in which the operation
$[t_1,t_2]^{\fnt{P}_{\Gamma}(\A)}\colon (A\times A)^{n}\to A\times A$ 
is  given by
\begin{multline*}
[t_1,t_2]^{\fnt{P}_{\Gamma}(\A)}
((a_1,b_1),\ldots,(a_{n},b_{n}))=\\ 
(t_1^{\A}(a_1,b_1,\ldots,a_{n},b_{n}), 
t_2^{\A} (a_1,b_1,\ldots,a_{n},b_{n})),
\end{multline*}
where  $n=n_{(t_1,t_2)}$ and   $ (a_1,b_1),\ldots,(a_{n},b_{n})\in A\times A$.
It is easy to check that
the assignment $\A\mapsto\fnt{P}_{\Gamma}(\A)$ (on objects)
and $h \mapsto h \times h$ (on morphisms) defines  a functor 
$\fnt{P}_{\Gamma}\colon \mathbb{V}(\CN) \to \mathbb{V}(\fnt{P}_\Gamma(\CN))$.
We shall also need the following notation.
Given  a set $X$ the map   $\delta^{X} \colon X\to X\times X$ 
is given by  
$\delta^{X} (x)=(x,x)$ and  $\pi_1^{X},\pi_2^{X}\colon X\times X\to X$
 denote  the projection maps.

We are ready to  recall  a key definition from~\cite[Section~3]{prod}, 
 where further details can be found.  
We  say that  $\Gamma$   \emph{duplicates} $\CN$
and that 
 $\duplicateV=\mathbb{V}(\fnt{P}_{\Gamma}(\CN))$ 
 is \emph{a duplicate of}  $\baseV$
if  the following conditions on~$\CN$ and~$\Gamma$ are satisfied:
\begin{newlist}
\item[(L)] for each $n$-ary operation symbol $f\in\Sigma$ and each $i\in\{1,2\}$ 
there exists an $n$-ary $\Gamma$-term $t$ (depending on ~$f$ and~$i$)  such that
$\pi^{N}_i\circ t^{\fnt{P}_{\Gamma}(\N)}\circ(\delta^{N})^{n}=f^{\N}$ 
for each $\N\in\CN$;  
\item[(M)] there exists a binary $\Gamma$-term $v$ such that 
$v^{\fnt{P}_{\Gamma}(\N)}((a,b),(c,d))=(a,d)$ 
for $\N\in\CN$ and $a,b\in N$;
\item[(P)] there exists a unary $\Gamma$-term $s$ such that 
$s^{\fnt{P}_{\Gamma}(\N)}(a,b)=(b,a)$ for  
$\N\in\CN$ and $a,b\in N$.
\end{newlist}

We  now present the Product Representation Theorem 
\cite[Theorem~3.2]{prod}.

\begin{thm}  \label{thm:GeneralProductEquiv}
Assume that  $\Gamma$ duplicates a class of algebras $\CN$
and let $\baseV=\mathbb{V}(\CN)$. 
Then the functor $\fnt{P}_{\Gamma}\colon \baseV\to \duplicateV$
sets up a categorical equivalence between  $\baseV$ and 
its duplicate $\duplicateV=\mathbb{V}(\fnt{P}_{\Gamma}(\CN))$.
\end{thm}

The classes of algebras arising in this section have prinicipally been varieties.  
In the next section we concentrate on singly-generated prevarieties.  
The following corollary 
 tells us how the class operators $\HSP$ and $\ISP$ 
behave with respect to  duplication.  
It is an almost immediate consequence of the fact that $\fnt{P}_{\Gamma}$ is a categorical equivalence; 
assertion (c) follows directly  from~(a) and~(b).

 \begin{coro}\label{Cor:ISPDuplicated}
Assume that $\Gamma$ duplicates a class of algebras $\CM$. 
The following statements hold
for each ${\A\in\mathbb{V}(\CM)}$: 
\begin{newlist}
\item[{\upshape (a)}]
$\HSP(\fnt{P}_{\Gamma}(\A))$ is categorically equivalent to $\HSP(\A)$.
\item[{\upshape (b)}]
$\ISP(\fnt{P}_{\Gamma}(\A))$ is categorically equivalent to $\ISP(\A)$.
\item[{\upshape (c)}]
If $\mathbb{V}(\A)=\ISP(\A)$ then $\mathbb{V}(\fnt{P}_{\Gamma}(\A))=\ISP(\fnt{P}_{\Gamma}(\A))$.
\end{newlist}
\end{coro}


\section{Natural duality and product representation}\label{Sec:NatDual}
It is appropriate to recall only in brief 
the theory of natural dualities as we  shall employ it.  
A textbook treatment is given in~\cite{CD98} and a summary 
geared to applications to distributive bilattices in \cite[Sections~3 and~5]{CPOne}.   

Our object of study in this section will be a prevariety $\CA$ generated by an
 algebra $\M$, so that $\CA = \ISP(\M)$.  
(Only in Section~\ref{Sec:prebilat} will we replace the single algebra~$\M$
by a family of algebras~$\CM$.  We shall then need to bring multisorted duality theory into play.)

Traditionally (and in \cite{CD98} in particular)~$\M$ is assumed to be finite.
This suffices for our applications in Section~\ref{NatViaDupl}.
However our application to bilattices with generalised conflation will 
depend on the more general theory presented  in~\cite{pig}.  
Therefore we shall assume that $\M$  can be equipped with a compact Hausdorff 
topology $\Tp$ with respect to which it becomes a topological algebra.
When~$\M$ is finite~$\Tp$ is necessarily discrete.

Our aim is to find  a second category~$\CX $
whose objects are topological structures of common type  and 
which is dually equivalent to~$\CA$ via functors $\D \colon \CA \to \CX $ and 
$\E \colon \CX  \to \CA$.  
Moreover---and this is a key feature  of a natural duality---we want each algebra  
$\A$ in $\CA$ to be  concretely representable 
as an algebra  of continuous structure-preserving maps from $\D(\A)$ 
(the \emph{dual space} of~$\A$) into~$\MT$, where $\MT \in \CX$
 has the same underlying set $M$ as does~$\M$. 
For this to succeed, some compatibility between the structures~$\M$ and $\MT$ will be necessary. 
We consider a topological structure $\MT = (M; G,R,\Tp)$ where 
\begin{itemize}
\item $\Tp$ is a topology on~$M$ (as demanded above);
\item $G$ is a set of operations on $M$, meaning that, for $g \in G$ of 
arity
$n \geq 1$, the map $g \colon \M^n \to \M$ is a 
continuous homomorphism
(any  nullary 
operation in $G$ will be identified with a constant in the type of~$\M$);
\item $R$ is a set of 
 relations on~$M$ such that 
 if $r\in R$ is $n$-ary   ($n \geq 1$) then $r$ is the universe of a 
topologically closed  
subalgebra $\mathbf r$ of~$\M^n$.
\end{itemize}
We refer to such a topological structure $\MT$ as an \emph{alter ego} for~$\M$ 
 and say that $\MT$ and $\M$ are \emph{compatible}.   
Of course. the topological conditions imposed on~$G$ and $R$ are trivially satisfied if~$M$ is finite.
(The general theory  in~\cite{CD98} allows 
an alter ego also to include partial operations,
but they do not arise in our intended applications.) 
We use $\MT$ to build a new category~$\CX $.  
We first consider structures of the same type as~$\MT$.
These have the form 
$\X = (X; G^\X ,R^\X, \Tp^\X)$ 
 where $\Tp^\X$ is a compact Hausdorff  topology 
and $G^\X$ and $R^\X$ are sets of  operations and relations on $X$
 in bijective correspondence   with those in $G$ and~$R$, with matching arities. 
Isomorphisms between such structures are defined in the obvious way.
For any non-empty set~$S$ we give $M^S$ the product topology 
and lift the elements of $G$ and~$R$ pointwise to~$M^S$.  
The \emph{topological prevariety} generated by~$\MT$ is 
$\CX  := \IScP(\MT)  $,  
the class of isomorphic copies of closed substructures of non-empty powers of $\MT$, 
with $^+$ indicating that the empty structure is included.
We make~$\CX $ into  a category by taking all continuous 
structure-preserving maps as the morphisms. 

As a consequence of the compatibility of $\MT$ and~$\M$, 
and the topological conditions imposed, 
the following assertions are true.
Let~$\A \in \CA$ and $\X \in\CX$.  Then  $\CA(\A,\M)$  may be 
seen as a closed
 substrucructure of $\MT^{A} $ and 
 $\CX (\X,\MT)$  as a subalgebra of $\M^{X}$.
We can  set up well-defined contravariant  hom-functors $\D \colon \CA \to \CX $
and $\E \colon \CX _\Tp \to \CA$; 
\begin{alignat*}{3}
&\text{on objects:} & \hspace*{2.5cm}  &\D \colon  \A \mapsto  \CA(\A,\M), 
  \hspace*{2.5cm}  \phantom{\text{on objects:}}&&\\
&\text{on morphisms:}  & &  \D \colon  x \mapsto - \circ x,&& \\
\shortintertext{and}
&\text{on objects:} & & \E  \colon  \X \mapsto  \CX (\X,\MT),
\phantom{\text{on objects:}}&&\\
&\text{on morphisms:}  & &\E  \colon  \phi \mapsto - \circ \phi,
\phantom{\text{on morphisms:}}&&
\end{alignat*}

The following assertions are part of the standard framework of natural duality theory.
Details can be found in \cite[Chapter~2]{CD98}; see also \cite[Section~2]{pig}.
Given $\A\in \CA$ and $\X \in \CX  $, 
we have  natural  evaluation maps 
$ e_{\A} \colon a \mapsto -\circ a$ and $\epsilon_{\X}\colon x \mapsto -\circ x$, 
with $e_\A\colon \A \to \E\D(\A)$ and  $\epsilon_\X \colon \X \to \D\E(\X)$.
Moreover  $(\D,\E,e,\epsilon)$ is a dual adjunction. 
Each of the maps $e_\A$ and $\epsilon_{\X}$ is an embedding.  
We say that~$\MT$ \emph{yields a duality on}~$\CA$, or simply that 
$\MT$ \emph{dualises}~$\M$, if each $e_\A$ is surjective, so that it is an 
isomorphism $e_\A\colon \A \cong \E\D(\A)$.   
A dualising alter ego~$\MT$ plays a special role in the duality it sets up:
it is the dual space of the free algebra on one generator in~$\CA$.   This fact is a consequence of compatibility.
 More generally, the free algebra generated by a non-empty  set~$S$ has dual space~$\MT^S$.

Assume that $\MT$ yields a duality on~$\CA$
and in addition that each $\epsilon_\X$ is  surjective and so an isomorphism.  
Then we say $\MT$ \emph{fully dualises} $\M$ or that the duality yielded by~$\MT$ is \emph{full}.  
In this case $\CA$ and~$\CX$ are dually equivalent.  
Full dualities are particularly amenable if they are \emph{strong};  
this is the requirement   that the 
alter ego be injective in the topological prevariety it generates.
 We do not need here to go deeply into the topic of strong dualities 
(see~\cite[Chapter~3]{CD98} for a full discussion)  
but we do note in passing that each of the functors~$\D$ and~$\E$ in a  strong  duality
 interchanges embeddings and surjections---a major virtue if a duality is 
 to be used to transfer algebraic problems into a dual setting.

We are ready to present our duality theorem for duplicated (pre)varieties.
Our notation is chosen to match that in Theorem~\ref{thm:GeneralProductEquiv}.

\begin{thm}[Duality Transfer Theorem]\label{thm:TransferNaturalDualities}
Let $\N$ be an algebra  and assume that $\Gamma$  duplicates~$\N$.
If the topological structure $\twiddle{\N}= (N;G,R,\Tp)$ yields a  duality on 
$\baseV= \ISP(\N)$ with dual category 
$\CY = \IScP(\NT)$, then $\twiddle{\N}^{2}$ yields a  
duality on $\duplicateV=\ISP(\fnt{P}_{\Gamma}(\N))$, 
again with $\CY$ as the dual category.
If the former duality is full, respectively strong, then the same is true of the latter. 
\end{thm}

\begin{proof}  
For the purposes of the proof we shall assume that~$N$, 
and hence also~$M$, is finite.  
It is routine to check that the topological conditions 
which come into play when~$N$ is infinite lift to the duplicated set-up.    

We claim  that $\NT^2$ acts as a legitimate  alter ego for $\M:= \fnt{P}_\Gamma(\N)$.
Certainly these structures  have the same universe, namely $N \times N$.
 It follows from the definition of the operations of  $\fnt{P}_\Gamma(\N)$  that
  $\fnt{P}_\Gamma(\mathbf r)$, whose universe is $r \times r$, 
  is a subalgebra of $(\fnt{P}_\Gamma (\N))^n$
whenever $r \in R$ is the universe of a subalgebra $\mathbf  r$ of $\N^n$.   
 But $R^{\NT^2}$ consists of the relations  $r \times  r$, for ${r \in R}$.
Likewise, an $n$-ary operation $g$ in $G$ gives rise to the same operation, 
{\it viz.}~$g \times g$,  of $\fnt{P}_\Gamma( \N)$ and in the structure~$\NT^2$. 
Hence  $g\times g$ is compatible with  $ \fnt{P}_{\Gamma}(\N)$.  

We now set up the functors for the existing  duality for $\ISP(\N)$ 
and for  the  duality sought for  $\ISP(\M)$.
Let  $\CX=\IScP(\twiddle{\N}^2)$.  
Then  $\CY = \IScP(\twiddle{\N}^2)=\CX$  too.  
Let  $\D_{\baseV}\colon\baseV\to\CY$ and $\E_{\CB}\colon\CY\to \baseV$
 be the functors determined by $\twiddle{\N}$ and 
 $\D_{\duplicateV}\colon\duplicateV\to\CX$ and $\E_{\duplicateV}\colon\CX\to \duplicateV$ 
 those  determined by $\twiddle{\N}^2$.  
 Since $\CY=\CX$, the functors  $\D_{\baseV}$ and $\D_{\duplicateV}$
 have a common codomain.

Let $\A\in\duplicateV$. 
By Corollary~\ref{Cor:ISPDuplicated}, 
we may assume that $\A=\fnt{P}_\Gamma(\B)$, 
for some $\B\in\baseV$.
By Theorem~\ref{thm:GeneralProductEquiv} and the definition 
of~$\fnt{P}_\Gamma$
on morphisms, 
\[
\duplicateV(\A,\fnt{P}_\Gamma(\N))=\fnt{P}_{\Gamma}(\baseV(\B,\N))=\{\,y\times y\mid y\in\baseV(\B,\N)\,\}.
\]
Let
$\alpha\in\E_{\duplicateV} \D_{\duplicateV}(\A) = \CX(\D_{\duplicateV}(\A),\twiddle{\N}^2)$.
 For $i=1,2$, define $\alpha_i\colon\D_{\baseV}(\B) \to \twiddle{\N}$
by $\alpha_i(y)=\pi_i^{N}(\alpha(y\times y))$ for  
$y\in\baseV(\B,\N)$.
It is straightforward to see that 
$\alpha_i\in\E_{\baseV}\D_{\baseV}(\B)$.
Therefore, for $i=1,2$,   there exists $b_i\in    B$  
such that $\alpha_i(y)=y(b_i)$ for  $y\in\D_{\baseV}(\B)$. 
We claim that $\alpha(x)=x(b_1,b_2)$ for  all
$x\in\duplicateV(\A,\fnt{P}_{\Gamma}(\N ))$.
We can write   $x=y\times y$ where $y\in\CB(\B,\N)$. 
Then
\begin{multline*}
\alpha(x)=\alpha(y\times y)=(\pi_1^N(\alpha(y\times y)),\pi_2^N(\alpha(y\times y)))\\
=(\alpha_1(y),\alpha_2(y))=
(y(b_1),y(b_2))=(y\times y)(b_1,b_2)
=x(b_1,b_2). 
\end{multline*}
This proves that
 $e_{\A} \colon \A \to  \E_{\duplicateV}\D_{\duplicateV}(\A)$ is surjective
for each $\A \in \duplicateV$, so that we do indeed have a duality for $\CA$ based
on the alter ego~$\MT= \NT^2$.

We now claim that if $\NT$ fully dualises~$\N$ 
then~$\MT$ fully dualises~$\M$.
To do this we shall show that the bijection 
$\eta \colon \D_{\baseV}(\B) \to \D_{\duplicateV}(\A)$,   
defined by $\eta(y)= y\times y$ for each $y\in \D_{\baseV}(\B)$, 
is an isomorphism (of topological structures)  from 
$\D_{\baseV}(\B)$ onto $\D_{\duplicateV}(\A)$,
where, as before, $\A=\fnt{P}_\Gamma(\B)$,
see 
\cite[Lemma 3.1.1]{CD98}. 
Let $r$ 
 be an $n$-ary relation in $\NT$. 
For 
$y_1,\ldots,y_n\in \D_{\baseV}(\B)$,
\begin{align*}
(y_1,\ldots,&y_n)\in r^{\D_{\baseV}(\B)}\\
&\Longleftrightarrow \forall a \in N \left((y_1(a),\ldots,y_n(a))\in r\right)\\
& \Longleftrightarrow \forall (a_1,a_2) \in M\left( ((y_1(a_1),y_1(a_2)),\ldots,(y_n(a_1),y_n(a_2)))\in r\times r\right)\\ 
&\Longleftrightarrow  (y_1\times y_1, \ldots,y_n\times y_n)\in (r \times r)^{\D_{\duplicateV}(\A)}.
\end{align*}
A similar argument applies to operations.

The map~$\eta$ has compact codomain and Hausdorff 
domain and hence is a homeomorphism provided $\eta^{-1}$
is continuous.  
To prove this it will suffice
to show that each map $\pi_b \circ \eta^{-1}$ is continuous, where 
$\pi_b$ denotes the projection from $\D_{\baseV}(\B)$, 
regarded as a subspace  of~$\NT^B$,
onto the $b$-coordinate, for $b\in B$. 
The map $\pi_{(b,b)}$ is defined likewise. 
Let~$U$ be open in~$N$.  
For  $y\times y\in\D_{\duplicateV}(\A)$,
\begin{multline*}
y\times y \in 
(\pi_b \circ \eta^{-1}  )^{-1}(U)  \Longleftrightarrow
            \pi_b(y) \in U 
\Longleftrightarrow  y(b) \in U \\
 \Longleftrightarrow (y\times y)(b,b) \in U \times U
\Longleftrightarrow \pi_{(b,b)}(y\times y) \in U\times U \\
 \Longleftrightarrow (y \times y) \in (\pi_{(b,b)})^{-1} (U \times U).
\end{multline*}
This proves the continuity assertion.

Finally, since $\NT$  is injective in $\CY$ if and only if  
$\NT^2$ is,  $\NT$ yields a strong duality on
 $\baseV$  if and only if  $\NT^2$ yields a strong duality   on~$\duplicateV$,
by \cite[Theorem~3.2.4]{CD98}.
\end{proof}

The proof of Theorem~\ref{thm:TransferNaturalDualities} is essentially routine, 
given the Product Representation Theorem.
The theorem should not be disparaged because it is easy to derive. 
Rather the reverse: almost all the dualities given in 
Section~\ref{NatViaDupl} are new, and obtained at a stroke.

Of course, though,  Theorem~\ref{thm:TransferNaturalDualities} is only useful 
when we have a (strong) duality to hand for the base class $\ISP(\N)$  we wish to employ.  
Nothing we have said about natural dualities so far tells us how to find an alter ego 
$\NT$ for~$\N$, or even whether a duality exists.  
Fortunately, simple and well-understood strong dualities exist for 
the base varieties $\ISP(\N)$ which support the miscellany of 
logic-oriented examples presented in Section~\ref{NatViaDupl}.   
In all cases considered there, 
$\N$ is a small finite algebra with a lattice reduct.  
Existence of such a reduct
guarantees dualisability  \cite[Section~3.4]{CD98}: a
brute-force alter ego $\NT = (N; \mathbb S (\N^2), \Tp)$ is available.  
However this default choice is likely to yield a tractable duality only 
when  $\N$ is very small.
 Otherwise the subalgebra lattice $\mathbb S(\N^2)$ is generally unwieldy.   
 Methodology exists for slimming down a given dualising  alter ego 
 to yield a potentially more workable duality (see \cite[Chapter~8]{CD98}),  
but it is preferable to obtain an economical duality from the outset. 
 This is often possible when $\N$ is a distributive lattice,
not necessarily finite:  in many such cases one can  apply the piggyback 
method which originated with  Davey and Werner
 (see~\cite[Chapter~7]{CD98} and \cite{pig}).  
We shall demonstrate its use in Section~\ref{sec:Conflation},
where we develop a  duality for double Ockham algebras, our base variety 
for studying generalised conflation.  
 
Against this background we can appreciate 
the merits  of  Theorem~\ref{thm:TransferNaturalDualities}.  
Suppose we have 
a class $\ISP(\M)$ (with $\M$ finite) which is expressible 
as a duplicate of a dualisable base variety $\ISP(\N)$.
Then $|M| = |N|^2$ and, on cardinality grounds alone,  
finding an  amenable duality directly for $\ISP(\M) $ 
could be challenging, 
whereas the chances are much higher that we have available, 
or are able to set up, a simple dualising alter ego $\NT$ for~$\N$.  
And then, given~$\NT$ we can immediately 
obtain an alter ego~$\MT$ for~$\M$, with 
the same number of relations and operations in~$\MT$ as in~$\NT$.

\section{Examples of natural dualities via duplication}\label{NatViaDupl}

We now present a miscellany of examples. 
All involve bilattices but, as noted earlier, the scope of
our methods is potentially wider.
We derive (strong) dualities for  certain (finitely generated) duplicated
varieties  given  in~\cite{prod} by calling on  
well-known (strong) dualities for their base varieties.
A catalogue of base varieties and duplicates is assembled in 
\cite[Appendix, Table~1]{prod}, with references to where
in the paper these examples are presented. 
Table~\ref{table:exs}  lists  alter egos for dualities for base varieties. 
 These dualities are discussed in \cite{CD98},  with 
 their sources attributed. 
Natural dualities for the indicated 
 duplicated varieties, also strong, can be read off from the table, 
 using the Duality Transfer Theorem.  
When specifying a generator for  each base variety,  we adopt
abbreviations for standard sets of operations:
\[
\mathcal F_ {\CL}  =\{ \lor, \land, 0,1\}, \qquad 
 \mathcal F_{\CB}  = \mathcal F_{\DM} =\mathcal F_{\class K}   = 
\mathcal F_ {\CL} \cup \{ \sim\};
\]  
we have elected to denote negation in Boolean algebras, De Morgan 
algebras and Kleene algebras by~$\sim$, 
to distinguish it from bilattice negation, $\neg$.

The top row of Table~1 should be treated as a prototype,  both algebraically and dually.
There the base variety is $\CCD$, the variety of bounded distributive lattices. 
 The duplicated variety in this case is the 
variety $\DB$ of distributive bilattices.  
It is generated (as a prevariety) by the
four-element algebra in~$\DB$.  
Full details of the natural duality for $\DB$ and its 
relationship to Priestley duality for the base variety $\CCD$
appear in~\cite{CPOne}.   
All the other examples in the table work in essentially
the same way. 
 The examples we list may be grouped into  two types.   
In one type, the duplicator~$\Gamma$ 
includes the set of terms used to duplicate the variety of bounded lattices to create
bounded bilattices, augmented with additional terms to capture other operations 
from terms in the base language;  this applies to~$\DB$ itself, 
to implicative bilattices, to distributive bilattices with conflation,  
to the varieties carrying Moore's operator.  
In examples of the second type the base-level generator $\N$
is already equipped with a (distributive) bilattice structure and $\Gamma$ 
includes all the terms used to create $\DB$ plus terms to create any extra operation
present in~$\N$.    
This is the situation with negation-by-failure.
 
For the natural dualities recorded in Table~1, we note that, apart 
from~$\CCD$,  the base variety in each case is De Morgan algebras or a subvariety thereof.   
The alter ego includes a partial order $\preccurlyeq$ known as the
\emph{alternating order} 
in~{\cite[Theorem~4.3.16]{CD98}}; in the case of~$\DM$, the relation
$\preccurlyeq$ on universe $\{0,1\}^2 $ of the four-element generator 
$\fourvar{\DM}$ is the knowledge order.  
The map~$g$ is the involution swapping the coordinates.

\begin{table} 
\begin{tabular}{| c | l | l   || l | }
\hline
 \multicolumn{3}{|c||}{}& 
 \\[-2.5ex] 
 \multicolumn{3}{|c||}{base variety and  its natural  duality}&
 \multicolumn{1}{|c|}{duplicate variety} 
\\[.5ex] 
\hline
&&&\\[-2.5ex] 
variety & \multicolumn{1}{|c|}{ generator}  &  \multicolumn{1}{|c||}{alter ego} &   \multicolumn{1}{|c|}{non-bilattice} \\
&&&  \multicolumn{1}{|c|}{operation  added}
  \\[.5ex]
\hline
&&&\\[-2.5ex] 
bounded 
&
$(\{0,1\}; \mathcal F_{\CL})$ 
&
$(\{0,1\};\leq,\Tp)$ & N/A \\
DL's&& \cite[\S 4.3.1]{CD98}&  
\\[.5ex] 
\hline
&&&\\[-2.5ex] 
& &   &implication, $\aaimp$ \\  
&&&
\hfill {\ \ \cite{AA1}, 
\cite[\S 2]{BJR11}} 
\\[1ex]
Boolean&$( \{0,1\}; \mathcal F_{\CB}) $ 
& $(\{0,1\};\Tp)$ &Moore's epistemic 
\\
algebras&&\cite[\S4.1.2]{CD98}& \quad 
operator, $L$ 
\hfill\cite{GinMod}\\[1ex]
&&& negation-by-failure, $\slash$\\
&&& \hfill\cite[\S3]{RF}\\[1ex]\hline
&&&\\[-2.5ex] 
De Morgan 
& $
(\{0,1\}^2; \mathcal F_{\DM})$ &
$ 
(\{0,1\}^2;\preccurlyeq,g,\Tp)$&conflation,  $-$
\\
 algebras&&\cite[\S4.3.15]{CD98} &  
\quad (with bounds)\quad  
\cite{JR}   \\[1ex]
\hline 

&&&\\[-2.5ex] Kleene   & $(\{0,a,1\}; \mathcal F_{\KL})$ & see \cite[\S4.3.9]{CD98}  &negation-by-failure, $\slash$\\
 algebras&&&   \hfill
  \cite[\S4]{RF}\\[.5ex] 

\hline
\end{tabular}
\caption{ Examples of natural dualities (bounded case)} \label{table:exs} 
\end{table}

Only simple modifications are needed to handle the case when the
language of a lattice-based variety does not include  lattice bounds as nullary 
operations.  It is an old result that Priestley duality for 
the variety $\CCD_u$ can be set up in much the same way as that for $\CCD$,
 with the dual category being  pointed Priestley spaces, as described in 
\cite[Section~1.2 and Subsection~4.3.1]{CD98}.  
 Natural dualities for duplicates  
of $\CCD_u$ are  derived 
from those for corresponding duplicates  of~$\CCD$ simply by adding 
to the alter ego nullary operations   $(0,0)$ and $(1,1)$.  
Compare with \cite[Section~4]{CPOne}, which provides 
a direct treatment of duality for $\DB_u$; here, even more than in the bounded case, 
we see the merit of the automatic 
process that Theorem~\ref{thm:TransferNaturalDualities} 
supplies.  
 A duality for $\DM_u$ (De Morgan lattices)
 is obtained by adding the top and bottom elements  for the partial order $\preccurlyeq$ to the alter ego for~$\DM$.
 Our transfer theorem then 
applies to unbounded distributive bilattices 
with conflation.

\section{Bilattices with generalised conflation}\label{sec:Conflation}

In this section we break new ground, both in relation to product representation 
and in relation to natural duality.   

The bilattice-based variety
$\DB_{\boldsymbol -}$ that we study---%
(bounded) distributive bilattices with generalised conflation---%
has not been considered before.  
Previous  authors who have studied product representation  when  conflation
is present   
have assumed that this operation  is an involution that commutes with negation
  (see~\cite[Theorem~8.3]{FKlkids},~\cite{BJR11}  and our treatment in~\cite[Section~5]{prod}).  
  We shall demonstrate that neither  assumption is  necessary
for the existence of a product representation.

Our focus in this paper is on developing theoretical tools.  
Nevertheless we should supply application-oriented reasons to justify  
investigating generalised conflation.  
We first note that it is often, but not always, 
 natural to assume that conflation be an involution.  
On the other hand, the  justification for the commutation condition
is less clear cut.  
Indeed,  both the original definition in~\cite{FKlkids} and that in \cite{OW15} 
exclude commutation,   and this is brought in only later.
In \cite[Section~3]{OW15} 
the emphasis is on truth values.  
The authors'  desired interpretation 
then leads them  to consider a special algebra $\text{SIXTEEN}_{3}$,  
in which the conflation operation does commute with negation.   
In \cite[Section 2]{G99} conflation is used to study 
(knowledge) consistent and exact elements of a lattice. 
The investigations  in  both 
\cite{OW15} and \cite{G99} are intrinsically connected to
 the product representation for bilattices with conflation. 
Our product representation would permit 
similar interpretations 
when commutation fails and/or conflation is not an involution.
In a different setting, conflation has been  used in \cite{FitAnnProg} 
to present an algebraic model of the logic system of revisions in databases, 
knowledge bases, and belief sets introduced in \cite{MaTr}. 
In this model the coordinates of a pair in a product representation of a bilattice are interpreted as the 
degrees of confidence for including in a database an item of information and  for excluding it.
Conflation  then models the transformation of information that reinterprets 
as evidence for inclusion whatever did not  previously  
count as evidence against, and vice versa. 
That is, conflation comprises two processes: 
given the information against (for) a certain argument, 
these capture information for (against) the same argument. 
In \cite{FitAnnProg} these two transformations coincide, and are mutually inverse.  
Our work on generalised conflation would allow these assumptions to be weakened
so facilitating a wider range of models.

The class $\DB_{\boldsymbol -}$ consists of algebras  of the form  
\[ 
\A=(A;\lor_t,\land_t,\lor_k,\land_k,\neg,-,0,1),
\]
where the reduct of $\A$ obtained by suppressing~$-$  belongs to~$\DB$ and~$-$
is an endomorphism of $\A_t$ and a dual endomorphism of $\A_k$.
Here we elect to include bounds. 
The variety $\DBC$ of (bounded) distributive bilattices with  conflation 
(where  by convention conflation  and  negation do commute)
is a subvariety of $\DB_{\boldsymbol -}$. 
 However $\DB_{\boldsymbol -}$ and $\DBC$  behave 
quite differently:  even though $\neg$ 
is an involution, $-$ is not. 
As a consequence the monoid these operations 
generate is not finite,
as is the case in $\DBC$. 
(We note that the unbounded case of generalised conflation could also be 
treated by making appropriate modifications to the above definition and throughout what follows.)

Our product representation for $\DB_{\boldsymbol -}$ uses as its  base 
variety the class $\cat{DO}$ of double Ockham algebras.
  This is a new departure as regards representations of bilattice expansions.
A \emph{double Ockham algebra} is a  $\CCD$-based algebras equipped with  two 
dual endomorphisms of the $\CCD$-reducts.
 An  Ockham algebra  carries  just one such operation. 
The variety~$\cat{O}$ of Ockham algebras, which includes  Boolean algebras, 
De Morgan algebras and Kleene algebras among its  subvarieties, 
has been exhaustively studied, both algebraically and via duality methods, 
as indicated by the texts \cite{BV,CD98} and many articles.  
The variety~$\cat{DO}$ is much less well explored.   
The remainder of the section is accordingly organised as follows.   
Proposition~\ref{prop:Ock-conflation} presents the product representation
 for $\DB_{\boldsymbol -}$ over the base variety~$\cat{DO}$.  
We then set  $\DB_{\boldsymbol -}$ aside  while we develop the  
theory of~$\cat{DO}$ which we need if we are to  apply
our  Duality Transfer Theorem to $\DB_{\boldsymbol -}$.  
 This requires us first to identify an algebra~$\M$ such that
$\cat{DO}=\ISP(\M) $ (Proposition~\ref{prop:Dock}).  
We then set up an alter ego~$\MT$  for~$\M$ and call on 
\cite[Theorem~4.4]{pig} to obtain a natural duality for~$\cat{DO}$
(Theorem~\ref{thm:DockDuality}).  
This is   then combined with Theorem~\ref{thm:TransferNaturalDualities}  
to arrive at a natural duality for 
$\DB_{\boldsymbol -} $ (Theorem~\ref{thm:GenConflationDuality}).

To motivate how we can realise $\DB_{\boldsymbol -}$ as a duplicate of
$\cat{DO}$  we briefly recall from \cite[Section~5]{prod}  
how $\DBC$ arises as a duplicate of~$\DM$.
We  adopt the notation  introduced in  \cite[Section~4]{prod}.
Let $\Sigma$ be a language and $f$ be an 
$n$-ary function symbol in $\Sigma$.  
For $m\geq n$ and $i_1,\ldots, i_n\in\{1,\ldots, m\}$ we denote by  
$f^{m}_{i_1\cdots i_n}$  the $m$-ary term 
$
f^{m}_{i_1\ldots i_n}(x_1,\ldots,x_m)=f(x_{i_1},\ldots,x_{i_n})
$.
We can  
capture the extra operation $-$ on 
the generator
$\Svar{\DBC}$
of~$\DBC$
 using the De Morgan 
negation~$\sim$,   
 combined with  coordinate-flipping: 
  the family of terms 
 $\Gammavari {\DBC}=\Gammavari {\DB}\cup\{(\sim^{2}_2,\sim^{2}_1)\}$ 
acts as a duplicator for~$\DM$ with~$\DBC$ as the duplicated
variety; here
$\Gammavari{\DB}$  duplicates bounded lattices.
(See \cite[Section~5]{prod} for an explanation as to why the form of the operations
in~$\DBC$ dictates that~$\DM$ should be used as the base variety.)

We now present our duplication result linking $\cat{DO} $ and $\DB_{\boldsymbol -}$.

\begin{prop}  \label{prop:Ock-conflation}
The set 
$\Gammavari{\DB_{\boldsymbol -}}=\Gammavari {\DB}\cup\{(f^2_2,g^2_1)\}$ 
duplicates 
$\cat{DO}$. 
Moreover,
$
\DB_{\boldsymbol -}
=\mathbb{V}\bigl( \fnt{P}_{\Gamma_{\DB_{\boldsymbol -}}}(\cat{DO}) \bigr)$,
where 
 $\Sigma_{\Gammavari{\DB_{\boldsymbol -}}}$ 
is identified 
with the language of $\DB_{\boldsymbol -}$. 
\end{prop}

\begin{proof}
Certainly 
 $\Gamma_{\DB_{\boldsymbol -}}$ duplicates $\cat{DO}$   because  
$(f^2_2,g^2_1)\in \Gamma_{\DB_{\boldsymbol -}}$
 and $\Gammavari{\DB}$ is a duplicate for 
$\Sigmavar {\CCD}$ on $\CCD$.

Now let $\A\in\DB_{\boldsymbol-}$. 
By the product representation of $\DB$ over $\CCD$, 
the bilattice reduct 
$\A_{\scriptscriptstyle \DB}
\cong \fnt{P}_{\Gamma_{\DB}}( \Lalg)$, for some $\Lalg \in \CCD$.  
We identify 
$A$ and $L\times L$ and  define $f,g\colon L\to L$ by
$f(a)=\pi_1(-(0,a))$ and $g(a)=\pi_2(-(a,0))$,
for $a\in L$. For  $a,b\in L$, 
\begin{align*}
g(a\vee b) &=\pi_2(-(a\vee b,0))=\pi_2(-((a,0)\vee_k(b,0))) 
		   =\pi_2(-(a,0)\wedge_k-(b,0)) \\ 
		&=\pi_2(-(a,0)\vee_k-(b,0))=\pi_2(-(a,0))\wedge \pi_2(-(b,0))
		  =g(a)\wedge g(b),  \\ 
g(a\wedge b) &=\pi_2(-(a\wedge b,0)) =\pi_2(-((a,0)\wedge_k(b,0)))\\
		     & =\pi_2(-(a,0))\vee \pi_2(-(b,0))=g(a)\vee g(b), \\
g(1) &=\pi_2(-(1,0))=\pi_2(1,0)=0, \quad 
g(0)=\pi_2(-(0,0))=\pi_2(1,1)=1,
\end{align*}
and similarly for $f$.
 Hence $\B=(L;\vee,\wedge,f,g,0,1)\in\cat{DO}$. 
Observe that
\allowdisplaybreaks
\begin{align*}
\pi_1(-(a,0))&=\pi_1(-((a,0)\vee_t (0,0)))
=\pi_1(-(a,0)\vee_t (1,1))=1;\\
\pi_2(-(0,b))&=\pi_2(-((0,b)\wedge_t (1,1)))
		  =\pi_1(-(0,b)\wedge_t (0,0))=0.\\
\shortintertext{Hence}
-(a,b)&=-((a,0)\vee_k(0,b))=-(a,0)\wedge_k -(0,b)\\
	&=(\pi_1(-(a,0)),\pi_2(-(a,0)))\wedge_k (\pi_1(-(0,b)),\pi_2(-(0,b)))\\
	&=(1,\pi_2(-(a,0)))\wedge_k (\pi_1(-(0,b)),0)\\
&=(1,g(a))\wedge_k (f(b),0)
	=(f(b),g(a))=[f^2_2,g^2_1](a,b).
\end{align*}
Therefore $\A\cong \fnt{P}_{\Gamma_{\DB_{\boldsymbol -}}}(\B)$.
\end{proof}

This theorem gives insight into the effect  of 
reinstating the assumptions customarily imposed on conflation and which 
we removed in passing 
from $\DBC$ to $\DB_{\boldsymbol -}$.
From the product representation for $\DB_{\boldsymbol -}$, it follows 
that $-$ is involutive if and only if~$f$ and~$g$ are.  
The resulting subvariety of $\DB_{\boldsymbol -}$ 
is a duplicate of double De Morgan algebras 
(that is, algebras in $\cat{DO}$ such that both unary operations are involutions).
Similarly,  $-$ commutes with~$\neg$ if and only if $f=g$. This time we obtain 
a subvariety of $\DB_{\boldsymbol -}$ which duplicates~$\cat{O}$.

We now want  to identify an (infinite) algebra which generates 
our base variety $\cat{DO}$ as a prevariety.
We take our cue from the variety $\cat{O}$ of Ockham algebras: 
 $\cat{O}$ is generated as a prevariety by an algebra 
$\M$  whose universe is $\{0,1\}^{\mathbb N_0}$, where 
$\mathbb N_0 = \{0,1,2.\ldots\}$;   
lattice operations and constants are obtained pointwise from 
the two-element bounded lattice and, identifying  the elements as infinite binary strings,
negation is given by a left shift followed by pointwise Boolean complementation on $\{ 0,1\}$.
 See for example \cite[Section~4]{pig} for details.  
 We may view the exponent
$\mathbb N_0$ as the free monoid on one generator~$e$,  with~$0$ as identity
and~$n$ acting as the $n$-fold composite of~$e$.

For $\cat{DO}$, analogously,  we first consider the free monoid 
$E=\{e_1,e_2\}^*$ on two generators $e_1$ and $e_2$
and identify it with the set of all finite words in the language with $e_1$ and $e_2$
 as function symbols, with the empty word corresponding to the identity element~$1$;
the monoid operation $\cdot$ is given by concatenation. 
For $s\in E$, we denote the length of~$s$ by $|s|$.  

For us, $\cat{DO}$ will serve as  a base variety.  
Accordingly we align our notation 
with that in Theorem~\ref{thm:TransferNaturalDualities}.
We now consider the algebra $\N$ with universe $\{0,1\}^E$ with  
lattice operations and constants given pointwise. 
The lattice $\{0,1\}^E$  is in fact a Boolean lattice, whose complementation operation we denote by~$c$.
 The dual endomorphisms $f$ and $g$ are given as follows.  
 For $a \in \{ 0,1\}^E$ we have 
 $f(a)(s)=c ( a (s\cdot e_1))$ and  $g(a)=c ( a(s \cdot e_2))$ for every $s\in E$. 
This gives us an algebra 
$
\N \coloneqq  ( \{0,1\}^E; \lor,\land, f,g ,0,1)\in~\cat{DO}$.

For future use we show how 
 to assign to each word $s\in E$ a unary term 
 $t_s$ in the language of $\cat{DO}$, as follows.  
If 
$s=1 $ (the empty word)  
then $t_s$ is the identity map;
 if $s=e_1\cdot s'$ then $t_s=f\circ t_{s'}$; and if $s=e_2\cdot s'$ then $t_s=g\circ t_{s'}$. 
Structural induction shows that  the term function $t_s^{\N}$ is given by
\begin{equation*} 
(t_s^{\N}(a))(e)=
\begin{cases}
	a(s\cdot e) &\mbox{if } |s| \mbox{ is even},\\
	1-a(s\cdot e) &\mbox{if } |s| \mbox{ is odd},
\end{cases}
\end{equation*}
for every $a\in N$ and $s\in E$.

\begin{prop}  \label{prop:Dock}
Let $\N$ be  defined as above. 
 Then $\cat{DO}= \ISP(\N)$.
\end{prop} 

\begin{proof}  
It will suffice to show that
 given any $\A \in \cat{DO}$ and any $a \ne b $ in~$\A$,  
there exists a $\cat{DO}$-morphism~$h$ from~$\A$ into~$\N$ such that $h(a) \ne h(b)$; 
 see \cite[Theorem 1.3.1]{CD98}.    
By the Prime Ideal Theorem
there exists a $\CCD$-morphism~$x$ from (the $\CCD$-reduct of) $\A$ 
into $\two$ with $x(a) \ne x(b)$.  
Define  $\varphi\colon \A \to \N$ by
\[
\varphi(c)(s) = 
\begin{cases}
	x(t_s(c))    & \mbox{ if } |s| \mbox{ is even,}\\
	1-x(t_s(c)) & \mbox{ if } |s| \mbox{ is odd},\\
\end{cases}
\]
  for $ c \in \A$ and $ s \in E$.
It is routine to  check that $\varphi$ is a $\CCD$-morphism 
which preserves~$f$ and~$g$. 
Finally, $\varphi(c)(1)= x(c)$, whence $\varphi (a) \ne \varphi(b)$.
\end{proof}

We now seek  a natural duality for $\cat{DO}$ which parallels 
that which is already known for the category $\cat{O}$ of Ockham algebras.
  Our treatment follows the same lines as that given for~$\cat{O}$ in \cite[Section~4]{pig}, 
  whereby a powerful version of the piggyback method is deployed.  
  (The duality for~$\cat{O}$ was originally developed by Goldberg \cite{Go83} 
  and re-derived as an early example of a  piggyback duality by Davey and Werner \cite{DWpig}.)
A general description of the piggybacking method and the ideas underlying it can be found in \cite[Section~3]{pig}.  
We wish to apply to $\cat{DO}$
a special case of \cite[Theorem~4.4]{pig}. 
   We first make some comments and establish notation.  
We piggyback over Priestley duality between  $ \CCD= \ISP(\two) $ and 
$\CP=\IScP(\twoT)$ (where $\two$ and $\twoT$ are the two-element objects 
 in~$\CCD$ and~$\CP$ with universe $\{0,1\}$, defined in the usual way). 
 We denote the hom-functors setting up the dual equivalence between~$\CCD$ and~$\CP$ by~$\fnt H$ and~$\fnt K$.  
 The aim is to find 
an element~$\omega \in \CCD(\N^\flat, \two)$ which, together with endomorphisms of~$\N$,  
captures enough information to build an alter ego~$\NT$ of~$\N$ 
which yields a full duality, in fact, a strong duality.

We now work towards showing that we can apply \cite[Theorem~4.4]{pig} 
to $\cat{DO}=\ISP(\N)$, where $\N$  is as defined above.  
We shall take 
$\omega  \colon \N \to \two$ to be the projection  map given by  $\omega (a) =a(1)$.
  We want to set up an alter ego
$\NT= (\{0,1\}^E; G,R,\Tp)$ so that in particular $\NT$ has a
Priestley space reduct $\NT^\flat$ such that $\omega \in \CP(\NT^\flat,\twoT)$.  
Moreover we need the  structure $\NT$ to be chosen in such a way 
that the conditions (1)--(3) in   
\cite[Theorem~4.4]{pig} are satisfied. 
We define  $\Tp$ to be  the product topology on  $N= \{0,1\}^E$ derived from the discrete
topology on $\{0,1\}$; this is compact and Hausdorff and 
 makes $\N$ into a topological algebra.  
 We now need to specify~$G$ and~$R$.
We would expect~$R$ to contain an order relation~$\preccurlyeq$ such that
$(\{0,1\}^E; \preccurlyeq ,\Tp) \in \CP$.  
For Ockham algebras---where 
one uses the free monoid on one generator as the exponent rather than~$E$---%
the corresponding order relation is the {\it alternating order} 
in which alternate coordinates are order-flipped; see \cite[Section~7.5]{CD98}
(and recall the comment about De Morgan algebras, a subvariety of~$\cat{O}$,
in Section~\ref{NatViaDupl}).  
 The key point  is that a 
composition of an even (respectively odd) number of order-preserving self-maps
on an ordered set  is order-preserving (respectively order-reversing). 
Hence the definition of $\preccurlyeq$ in Lemma~\ref{lem:Cond3} is entirely natural.

\begin{lem}\label{lem:Cond3}
Let $\N$ be  as above.
Then  $\preccurlyeq$, given by 
\[
a \preccurlyeq b \Longleftrightarrow  \forall z \in E\,
\begin{cases}
   a(s) \leq b(s) & \text{if }|s| \mbox{ is even,} \\
   a(s) \geq b(s) & \text{if } |s| \mbox{ is odd,}
\end{cases}
\] 
is an order relation making $(\{ 0,1\}^E; \preccurlyeq, \Tp)$ a Priestley space.  
Moreover $\preccurlyeq$ is the universe of a subalgebra of~$\N^2$ and this 
subalgebra is the unique maximal subalgebra of 
$
(\omega,\omega)^{-1}(\leq)=
 \{\, (a,b)\in N^2 \mid \omega(a) \leq \omega (b)
\, \}$.
\end{lem}

\begin{proof}  
Each of $\twoT$ and the structure
  $\twoT^\partial $  (that is, $\twoT$ with  the  order reversed)
is a Priestley space. 
It follows that the topological structure $(\{0,1\}^E; \preccurlyeq, \Tp)$ 
is a product of Priestley spaces and so itself a Priestley space.

Take  $a,b,c,d$ in $N$ such that $a\preccurlyeq b$ and $c\preccurlyeq d$ and let $s\in E$. 
Then
\begin{alignat*}{2}
(a\wedge c)(s) &=a(s)\wedge c(s)\leq b(s)\wedge d(s)=(b\wedge d)(s) \qquad &&\text{if }
|s| \mbox{ is even,}\\  
(a\wedge c)(s) &=a(s)\wedge c(s)\geq b(s)\wedge d(s)=(b\wedge d)(s) &&
\text{if } |s| \mbox{ is odd.}
\end{alignat*}
Hence $a\wedge c\preccurlyeq b\wedge d$. Similarly $a\vee c\preccurlyeq b\vee d$.
Also $0\preccurlyeq 0$ and $1 \preccurlyeq 1$.
If $|s|$ is even, 
$f(a)(s)=(c\circ a \circ e_1)(s)= 1-(a(e_1\cdot s))\leq 1-(b(e_1\cdot s))=f(b)(s)$,
since $a\preccurlyeq b$ and $|e_1\cdot s|$ is odd. 
Similarly, if $|s|$ is odd then  $f(a)(s)\geq f(b)(s)$.  
Therefore $f(a)\preccurlyeq f(b)$. 
Likewise  $g(a)\preccurlyeq g(b)$. 
Thus   $\preccurlyeq$ is indeed the universe of a subalgebra of $\N^2$.

Now let $r$ be the universe of a subalgebra of $\N^2$ maximal with respect to inclusion in
$(\omega,\omega)^{-1}(\leq)$.  
Then, with $t_s$ as defined earlier for $s\in E$, we have  
\begin{multline*}(a,b)\in r  
\ \Longrightarrow \ (\forall s\in E )\bigl((t_s(a),t_s(b)) \in r\bigr)
            \     \Longrightarrow \ (\forall s \in E) \bigl( t_s(a) \leq t_s(b)\bigr)\\
                 \Longrightarrow  (\forall s \in E)(\forall e \in E)
 \bigl( t_s(a)(e) = 1 \implies   t_s(b)(e) = 1\bigr).
\end{multline*}
But
\[
(t_s(a))(e)=\begin{cases}
a(s\cdot e) &\mbox{if } |s| \mbox{ is even,}\\ 
1-a(s\cdot e) &\mbox{if } |s| \mbox{ is odd.}
\end{cases}
\]
We deduce that $r$ is a subset of $\preccurlyeq$.  
In addition $a\preccurlyeq b$ implies $\omega (a) \leq \omega (b)$:
consider $s = 1$.  
Maximality of~$r$ implies that $r$ equals $\preccurlyeq$.
Consequently $\preccurlyeq$ is the unique maximal subalgebra contained in
$(\omega,\omega)^{-1}(\leq)$.
 \end{proof}

We now introduce the operations we shall include in our alter ego~$\NT$. 
 Let  the map $\gamma_i \colon E \to E$ be given by
$\gamma_i(s) = s\cdot e_i$.   
Then  we can  define an endomorphism $u_i$  of $\N$ 
by $u_i(a) = a \circ \gamma_i $,  for $i=1,2$.  
These maps are continuous with respect to the topology $\Tp$ we have put on~$N$.
We define
\[
\NT:= (\{0,1\}^E; u_1,u_2, \preccurlyeq, \Tp).
\]
Then $\NT$ is compatible with~$\N$.  
 We let 
$\CY \coloneqq  \IScP(\NT)$ be the topological prevariety generated by~$\NT$
and by $^\flat$  the forgetful functor from~$\CY$ into~$\CP$  which suppresses
the operations $u_1$ and~$u_2$.   We note that now $\omega$, as defined earlier, 
may be seen to belong to  $\CCD (\N^\flat, \two) \cap 
\CP(\NT^\flat,\twoT)$.
The following  two lemmas concern the interaction of $\N$, $\NT$ and~$\omega$ 
as regards separation properties.

\begin{lem}  \label{lem:Cond1}
Assume that $\N$, $\NT$ and $\omega$ are defined as above. 
 Then, given $a\ne b$ in $N$,
 there exists a unary term~$u$ 
 in the language 
of $(N; u_1,u_2)$  such that $\omega(u(a))\ne \omega(u(b))$.
\end{lem}

\begin{proof}  
Let $a \ne b \in \N$.  
There exists $s \in E$ with $s \ne 1$ such that $a(s) \ne b(s)$.  
Write $s$ as a concatenation 
$e_{i_1}\cdot \, \, \cdots \, \, \cdot e_{i_n}$, where $i_1,\ldots,i_n\in\{1,2\}$. 
For each   
$j=1,
\dots,n$, there is an associated unary term
$u_j$ such that,  for all~$w \in E$,
\[
(u_{i_j}(a))(w) = (a \circ \gamma_{i_j}) (w) = a(w \cdot e_{i_j}).
\]
Write   $u_{i_n}\circ\ldots\circ u_{i_1}$ as~$u_s$.  
Then $u_s(c)(1) = c(s)$ for all $c \in \N$ and hence  
\[
(\omega \circ u_s)(a) = u_s(a)(1) = a(s) \ne b(s) = u_s(b)(1) = (\omega  \circ u_s)(b). \qedhere  
\]
\end{proof}

\begin{lem}\label{lem:Cond2} 
If  $a\not\preccurlyeq b$
in $\NT^\flat$, then there exists a unary term function~$t$ of ~$\N$ 
such that $\omega(t(a))= 1$ and $\omega(t(b))= 0$.
\end{lem}

\begin{proof}
We have
\[
a \not \preccurlyeq b \Longleftrightarrow \exists s \in E
\begin{cases}
     a(s) = 1 \ \& \ b(s) = 0 &\text{if   $|s|$ is even},\\
     a(s) = 0 \ \& \ b(s)= 1 & \text{if  $|s|$ is odd}.
\end{cases}
 \]
When $|s|$ is even,
$\omega(t_s(a))=t_s(a)(1)=a(s)=1$ and $\omega(t_s(b))=t_s(b)(1)=b(s)=0$.
 Similarly, if $|s|$ is odd,
$\omega(t_s(a))=t_s(a)(1)=c\circ a(s)=1-a(s)=1$ and $\omega(t_s(b))=t_s(b)(1)=c\circ b(s)=1-b(s)=0$. 
\end{proof}

\begin{thm}[Strong Duality Theorem for Double Ockham Algebras]
\label{thm:DockDuality}  
Let
$\N= (\{0,1\}^E; \lor,\land,f,g,0,1)$
 and
$\NT= (\{0,1\}^E; u_1,u_2,\preccurlyeq ,\Tp)$  
be as defined above.   
Let $\omega \in \CCD (\N^\flat, \two) \cap \CP(\NT^\flat,\twoT)$ 
be given by evaluation at~$1$, the identity of the monoid~$E$.  
Let $\fnt D \colon \cat{DO} \to \cat Y$ and $\fnt E \colon \cat Y \to \cat{DO}$ 
be the hom-functors:  $\fnt D\coloneqq \cat{DO}(-,\N)$ and 
$\fnt E \coloneqq  \CY (-,\NT)$. 
Then 
$\NT$ strongly dualises $\N$, that is, $\fnt D$ and $\fnt E$ 
establish a strong duality between~$\cat{DO}$ and~$\CY$. 
Moreover 
\[
  \fnt D(\A)^\flat \cong \fnt H(\A^\flat) \text{  in }\CP\quad \text{and} \quad
  \fnt E(\Y)^\flat \cong \fnt K(\Y^\flat) \text{  in }\CCD,
 \]
 for $\A \in \cat{DO} $ and $\Y \in \CY$, where the isomorphisms are  set up by
$\Phi_\omega^\A \colon x \mapsto \omega \circ x $,  for $x \in \fnt D(\A)$, and
$\Psi_\omega^\Y \colon \alpha \mapsto\omega \circ \alpha$, for $\alpha \in 
\fnt E(\Y)$.
\end{thm}

\begin{proof}  
We simply need to confirm that the conditions 
of~\cite[Theorem~4.4]{pig} are satisfied.  
We have everything set up to ensure that all the functors work as the theorem requires.  
In addition Lemmas~\ref{lem:Cond3}--\ref{lem:Cond2}
tell us that Conditions~(1)--(3) in the  theorem are satisfied.  
\end{proof}

Some remarks are in order here.  
We stress that 
it is critical that we could find a map $\omega$ which acts as a morphism 
both on the algebra side and on the dual side, and has 
the  separation properties set out in Lemmas~\ref{lem:Cond1} and~\ref{lem:Cond2}.
We also observe that for our  application
of \cite[Theorem~4.4]{pig}, its Condition~(3) is met in 
a simpler way than the theorem allows for: the special  form of the 
$f,g$ 
({\it viz.}~dual endomorphisms  with respect to the bounded lattice operations) 
that forces 
 $(\omega,\omega)^{-1}(\leq)$ to contain just one maximal subalgebra.

We should comment  too  on how 
our natural duality for $\cat{DO}$
relates to  a Priestley-style duality for~$\cat{DO}$.  
The latter
can be set up in just the same way as that for~$\cat{O}$ originating in~\cite{Urq}. 
 This duality is an enrichment of that between~$\CCD$ and~$\CP$,
  whereby~$f$ and~$g$ are captured on the dual side 
  via a pair of order-reversing continuous maps $p$ and~$q$, 
  and morphisms are required to preserve these maps.   
Theorem~\ref{thm:DockDuality} tells us  that, for any $\A \in \cat{DO}$,  
there is an isomorphism 
between the Priestley space reduct $\fnt D(\A)^\flat$ 
of the natural dual of $\A\in \cat{DO}$ and the 
Priestley dual $\fnt H(\A^\flat)$ of the $\CCD$-reduct of~$\A$.  
Both these Priestley spaces carry additional structure:  
$u_1$ and~$u_2$ in the former case and~$p$ and~$q$ in the latter.  
When the reducts of 
the natural and Priestley-style dual spaces of  the algebras are identified 
these pairs of maps coincide.
 Thus the two dualities for~$\cat{DO}$ are essentially the same 
 and one may toggle between them at will. 
  We have a new example here  of a `best of both worlds' scenario,
   in which we have both the advantages of a natural
duality 
and the benefits, pictorially,  
of  a duality based on Priestley spaces.  
See \cite[Section~3]{CPcop},  \cite[Section~6]{CPOne}  
and \cite[Section~4]{pig}
for earlier recognition of  occurrences of  this phenomenon: 
 other varieties for which it arises are De Morgan algebras and Ockham algebras.  
 In general it is not hereditary:  it fails to occur for Kleene algebras, for example.

Combining  our results we arrive at our duality for the variety 
$\DB_{\boldsymbol -}$.

\begin{thm}[Strong Duality Theorem for Bounded Distributive Bilattices with Generalised Conflation]  
\label{thm:GenConflationDuality}
Let $\NT= (\{0,1\}^E; u_1,u_2, \preccurlyeq, \Tp)$ be as in Theorem~{\upshape\ref{thm:DockDuality}}.  
Then $\NT \times \NT$ yields a strong duality on $\DB_{\boldsymbol -}$. 
Moreover the dual category for this duality is~$\CY:= \IScP(\NT)$ 
 which may, in turn, be identified with the category $\CP_{\cat{DO}}$ of double Ockham spaces. 
 \end{thm}

 To illustrate the rewards derived from  a natural duality for $\fnt F_{\DB_{\boldsymbol -}}$, 
we highlight the simple description of free objects that follows  from
Theorem~\ref{thm:GenConflationDuality}: for 
a non-empty set~$S$,  
the free algebra $\fnt F _{\DB_{\boldsymbol -}}(S)$
on~$S$ has  $(\NT^2)^S$ as its natural dual space. 
Hence
$\fnt F _{\DB_{\boldsymbol -}}(S)$ can be identified with 
the family of continuous  structure-preserving 
maps from $(\NT^2)^S$ into $\NT^2$, with the operations defined pointwise.  
(Recall the remark on free algebras in Section~\ref{Sec:NatDual}.)

\section{Dualities for pre-bilattice-based varieties}   \label{Sec:prebilat}

In this final section we consider dualities for pre-bilattice-based varieties. 
Here we call on the adaptation of the product representation theorem given in~\cite[Theorem 9.1]{prod}.
Hitherto in this paper  we have worked with dualities for prevarieties
of the form~$\ISP(\M)$, thereby encompassing dualities for many classes
of interest in the context of bilattices.    
However when we drop negation
and so move from bilattices to pre-bilattices the situation changes and we 
encounter classes of the form~$\ISP(\CM)$, 
where~$\CM$ is a finite set of algebras over a common language. 
For example, for 
distributive pre-bilattices $\CM$ consists of a pair of two-element algebras,
one with truth and knowledge orders equal, the other with these as order duals.  
Fortunately a form of natural duality 
theory exists which is applicable to classes of the form~$\ISP(\CM)$;  
this  makes use of multisorted structures on the dual side.
So in this section we shall consider dualities for pre-bilattice-based 
varieties.   
As a starting point we have the 
treatment of distributive pre-bilattices given in \cite[Sections~9 and~10]{CPOne};
a self-contained summary of the rudiments of  multisorted duality
theory can also be found there or see \cite[Chapter~7]{CD98}.

We first recall how \cite[Theorem~9.1]{prod} differs from 
Theorem~\ref{thm:GeneralProductEquiv}.   
We start from a base class $\mathbb V(\CN)$, where $\CN$
 is a class of algebras over a common  
 language~$\Sigma$.
 Let~$\Gamma$ and  $\fnt P_\Gamma (\CN)$ be as in Section~\ref{Sec:Prelim}.  
Negation in a product bilattice links the two factors, and condition~(P) from
the definition of duplication by~$\Gamma$ reflects this.  
In the absence of negation,~(P) is dropped and the following condition is substituted:
\begin{newlist}
\item[(D)] for $(t_1,t_2)\in\Gamma$ with $n_{(t_1,t_2)}=n$, there exist  $n$-ary 
$\Sigma$-terms $r_1$ and $r_2$ such that
$t_1(x_1,\ldots,x_{2n}) =r_1(x_{1},x_3,\ldots, x_{2n-1})$ and 
$t_2(x_1,\ldots,x_{2n}) =r_2(x_{2},x_4,\ldots, x_{2n})$.
\end{newlist}
  A product algebra associated with~$\Gamma$ now takes the form  
\[
\PP\odotG{\Gamma}\Q
=( P\times Q; \{ [t_1,t_2]^{\PP\odotG{\Gamma}\Q}\mid (t_1,t_2)\in\Gamma \} ),
\]
where  $\PP,\Q$ belong to the base variety $\baseV=\mathbb{V}(\CN)$. 
This construction is used to define a functor $\odotG{\Gamma}\colon\baseV\times\baseV\to \duplicateV $ as follows:
 \begin{alignat*}{3}
&\text{on objects:} & \hspace*{2.5cm}  & (\PP,\Q) \mapsto \PP\odotG{\Gamma}\Q, 
  \hspace*{2.5cm}  \phantom{\text{on objects:}}&&\\
&\text{on morphisms:}  & &  \odotG{\Gamma}(h_1,h_2) (a,b)=(h_1(a),h_2(b)).
\end{alignat*}

\begin{thm}  {\upshape{\cite[Theorem~9.3]{prod}}} \label{thm:ICP}
Let $\CN$  be a class of $\Sigma$-algebras and 
let $\Gamma$ a set of pairs of $\Sigma$-terms satisfying
{\upshape (L), (M)} and {\upshape(D)}.  
Let $\baseV=\mathbb{V}(\CN)$. 
Then the functor $\odotG{\Gamma}\colon \baseV\times \baseV\to\duplicateV$,
sets up a categorical equivalence between 
$\baseV\times \baseV$ 
and $\duplicateV=\mathbb{V}(\{\PP\odotG{\Gamma}\Q\mid \PP,\Q\in\mathbb{V}(\CN)\})$.
\end{thm}

We  move on to consider dualities for duplicated varieties.
For simplicity we shall first assume that the base variety
$\baseV = \ISP(\N)$ has a single-sorted duality 
with  alter ego $\twiddle{\N}= \left( N ; G,R,\Tp\right)$. 
Our next task is to   determine
 a set of generators for $\duplicateV$ as a prevariety.
We denote the trivial algebra by~$\T$. For 
$\C\in \baseV$ let $f^{*}_{C}\colon C\to T$ 
be the unique homomorphism from $\C$ into $\T$. 

\begin{lem}
If $\CB = \ISP(\N)=\mathbb{V}(\N)$ for some algebra $\N$, then
\[
\duplicateV=
\mathbb{V}(\{\PP\odotG{\Gamma}\Q\mid \PP,\Q\in\ISP(\N)\})=
\ISP(\N\odotG{\Gamma} \T,  \T \odotG{\Gamma} \N). 
\]
\end{lem}
\begin{proof}
Let $\A\in\duplicateV$ and $a\neq b\in A$. 
By Theorem~\ref{thm:ICP}, 
we may assume that there exist $\B,\C\in\baseV$ such that $\A= \B\odotG{\Gamma}\C$. 
Let 
$a_1,b_1\in B$ and $a_2,b_2\in C$ such that 
$a=(a_1,a_2)$ and $b=(b_1,b_2)$. 
By simmetry we may assume that $a_1\neq b_1$. 
Then there exists a homomorphism 
$h\colon \B\to \N$ such that $h(a_1)\neq h(b_1)$.  
Now 
$h\odotG{\Gamma} f^{*}_{C}\colon \B\odotG{\Gamma} \C \to \N\odotG{\Gamma} \T$ 
 is such that
\[
(h\odotG{\Gamma} f^{*}_{C})(a)=(h(a_1),f^{*}_{C}(a_2))
\neq 
(h(b_1),f^{*}_{C}(b_2))=(h\odotG{\Gamma} f^{*}_{C})(b).\qedhere\]
\end{proof}

Let $\CM=\{\N\odotG{\Gamma} \T,  \T \odotG{\Gamma} \N\}$.
We now `double up'  $\twiddle{\CN}$ in the obvious way. 
Let
$\twiddle{\CN}\uplus\twiddle{\CN}= \left(  N_1\dotcup N_2; G_1,G_2,R_1,R_2,\Tp\right)$,  
 based on disjointified universes $N_1$ and $N_2$, such that 
 $(N_i;G_i,R_i,\Tp{\upharpoonright}_{N_i})$ is isomorphic to 
 $\twiddle{\CN}$ for $i=1,2$. 
Identify $N_1$ with $N\times T$  and $N_2$ with $T\times N$ 
and define $\twiddle{\CM}=\twiddle{\CN}\uplus\twiddle{\CN}$.

We now present 
our transfer theorem for natural dualities associated with 
Theorem~\ref{thm:ICP} (the single-sorted case).
Its proof is largely a diagram-chase with functors.  
Below, $\fnt{Id}_{\cat C}$ denotes  
the identity functor on a category $\cat C$
and~$\cong$ is used to denote 
natural isomorphism.

\begin{thm} \label{thm:transfer-D}
Let $\N$ be a  $\Sigma$-algebra and 
assume that $\Gamma$ satisfies {\upshape (L), (M)} and {\upshape (D)} 
relative to $\N$.
Assume that $\twiddle{\N}= (N;G,R,\Tp)$ yields a 
 duality on $\baseV= \ISP(\N)=\mathbb{V}(\N)$ with dual category
 $\CY = \IScP(\twiddle{\N})$.  
Let~$\CM$ and~$\twiddle{\CM}$ be defined as above.
 Then $\twiddle{\CM}$ yields a multisorted duality for  
$\duplicateV= \ISP(\CM)=\mathbb{V}(\PP\odotG{\Gamma}\Q\mid \PP,\Q\in\mathbb{V}(\CN))$  
for which the dual category  is $\CX \cong \CY \times\CY$.  
If the  duality for~$\baseV$  is full, respectively strong, 
then the same is true of that for~$\duplicateV$. 
\end{thm}

\begin{proof}
Let $(\X_1,\X_2)\in\CY\times\CY=\IScP(\twiddle{\N})\times \IScP(\twiddle{\N})$. 
We identify this structure with
 $\X_1\uplus\X_2=(X_1\dotcup X_2;\,G_1,G_2,R_1,R_2,\Tp)$,
 where as before $\dotcup$ denotes disjoint union and 
the topology $\Tp$ 
is the union  of $\Tp_1$ and $\Tp_2$.
Morphisms in $\CX$ are maps 
$f\colon X_1 \dotcup  X_2 \to Y_1\dotcup Y_2 $ 
that respect the structure and 
are 
such that $f(x)\in Y_i$ when $x\in X_i$ and $i\in\{1,2\}$. 
 Hence    the assignment: 
\begin{alignat*}{3}
&\text{on objects:} & \hspace*{2.5cm}  &(\X_1,\X_2)\mapsto  \X_1\uplus\X_2 , 
  \hspace*{2cm}  \phantom{\text{on objects:}}&&\\
&\text{on morphisms:}  & & (f,g)  \mapsto f\dotcup g 
\end{alignat*} 
sets up a categorical equivalence, $\uplus$.
Let 
 $\fnt F\colon \CX\to \CY\times \CY$ denote 
its inverse.

\begin{figure}  [ht]
\begin{center}
\begin{tikzpicture} 
[auto,
 move up/.style=   {transform canvas={yshift=1.9pt}},
 move down/.style= {transform canvas={yshift=-1.9pt}},
 move left/.style= {transform canvas={xshift=-2.5pt}},
 move right/.style={transform canvas={xshift=2.5pt}}] 
\matrix[row sep= 1.3cm, column sep= 1.25cm]
{ 
 \node (B) {$\baseV\times \baseV$}; &&\node (Y){$\CY\times \CY$}; & \node (BB) {$\baseV\times \baseV$}; &&\node (YY){$\CY\times \CY$};
 
\\ 
 \node (A) {$\duplicateV$}; && \node (X) {$\CX$}; &  \node (AA) {$\duplicateV$}; && \node (XX) {$\CX$};
\\
};

\draw [-latex ] (B) to node    {$\D_{\baseV}\times\D_{\baseV}$}(Y);
\draw [latex-] (B) to node [swap]  {$\fnt{C}$} (A);
\draw [-latex] (A) to node  [swap]  {$\D_{\duplicateV}$}(X);
\draw [-latex] (Y) to node  {$\uplus$}(X);
\draw [latex-] (BB) to node {$\E_{\baseV}\times\E_{\baseV}$}(YY);
\draw [-latex] (BB) to node  [swap] {$\odotG{\Gamma}$}(AA);
\draw [latex-] (AA) to node [swap] {$\E_{\duplicateV}$}(XX);
\draw [latex-] (YY) to node {$\fnt{F}$} (XX);

\end{tikzpicture}
\end{center}
\vspace{-15pt}
\caption{Natural duality by duplication}\label{Fig:NatDual}
\end{figure}
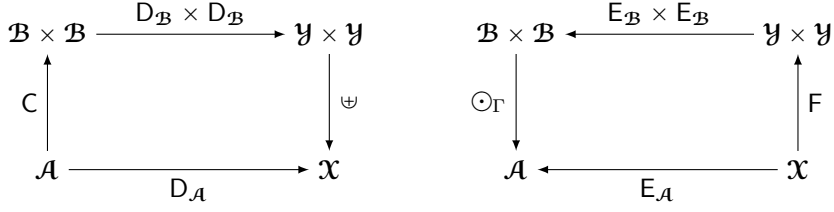 

Identify 
$\N\odotG{\Gamma}\T$ and $\T \odotG{\Gamma}\N$ with $N_1$ and $N_2$ respectively. 
One sees 
that $\twiddle{\CM} \coloneqq 
\twiddle{\CN}\uplus \twiddle{\CN}=
\left(  N_1\dotcup 
N_2; G_1,G_2,R_1,R_2,\Tp\right)$ is a  legitimate  alter ego for $\CM$.
Let $\D_{\baseV}\colon\baseV\to\CY$ and $\E_{\baseV}\colon\CY\to \baseV$,
and 
 $\D_{\duplicateV}\colon\duplicateV\to\CX$ and $\E_{\duplicateV}\colon\CX\to \duplicateV$ 
 be the hom-functors   
 determined by  $\twiddle{\CN}$ and $\twiddle{\CM}$ respectively.  
 By Theorem~\ref{thm:ICP}, there exists a functor 
  $\fnt C\colon \duplicateV\to \baseV\times\baseV$ 
  that together with 
 $\odotG{\Gamma}\colon  \baseV\times\baseV \to \duplicateV$ determines
 a categorical equivalence. 
 Take $\A,\B\in\baseV$ and let 
\[
\D_{\duplicateV}(\A\odotG{\Gamma}\B)=(X_1\dotcup 
X_2;G_1,G_2,R_1,R_2,\Tp).
\]
Again by Theorem~\ref{thm:ICP}, 
\[
X_1  =\duplicateV(\A\odotG{\Gamma}\!\B,\N\odotG{\Gamma}\!\T)= \{(h_A,f_B^{*}) \mid h_A\in\baseV( \A, \N )\} 
=\baseV( \A, \N )\times \{f^{*}_B\},\\ 
\]
and likewise $X_2  = \{f^{*}_A\} \times  \baseV( \B,\N)$.

 For an $n$-ary relation $r\in R$,
  let $r_i^{\A\odotG{\Gamma}\B}$ be the corresponding relation in 
 $R_i^{\A\odotG{\Gamma}\B}\subseteq X_i^n$  ($i=\{1,2\}$).
So $(h_1,\ldots, h_n)\in r_1^{\A\odotG{\Gamma}\B} $ if and only if 
 $h_i=(g_{i},f^{*}_{B})\in \baseV( \A, \N )\times \{f^{*}_B\} $ for 
$i\in\{1,\ldots, n\}$ and $(g_1,\ldots,g_n)\in r^{\A}$.
 Similarly, a tuple $(h_1,\ldots, h_n)$ belongs to 
$r_2^{\A\odotG{\Gamma}\B} $ if and only if 
 $h_i=(f^{*}_{A},g_{i})\in  \{f^{*}_A\} \times  \baseV( \B, \N ) $ for 
 $i\in\{1,\ldots, n\}$ and $(g_1,\ldots,g_n)\in r^{\B}$. 
The same argument applied to $G$ proves that 
$(X_1;G_1,R_1,\Tp{\upharpoonright}_{X_1})$ and $(X_2;G_2,R_2,\Tp{\upharpoonright}_{X_2})$ 
are isomorphic to $\D_{\baseV}(\A)$ and $\D_{\baseV}(\B)$, respectively.
 Thus $\fnt F(\D_{\duplicateV}(\A\odotG{\Gamma}\B))$ is isomorphic to
  $(\D_{\baseV}(\A),\D_{\baseV}(\B))$  in $\CY\times\CY$.
 Moreover, it is easy to see that the assignment 
  $\fnt F(\D_{\duplicateV}(\A\odotG{\Gamma}\B))\mapsto(\D_{\baseV}(\A),\D_{\baseV}(\B))$ 
 determines a natural isomorphism between $\fnt F\circ \D_{\duplicateV}\circ \odotG{\Gamma}$ 
 and $\D_{\baseV}\times\D_{\baseV}\colon \baseV\times\baseV \to \CX\times\CX$.

Similarly, for each  $(\X,\Y)\in \CX\times\CX$, 
\begin{multline*} 
\E_{\duplicateV}(\X\uplus\Y)
=(\E_{\baseV}(\X)\odotG{\Gamma}\T) \times (\T\odotG{\Gamma}\E_{\baseV}(\Y))\\
\cong(\E_{\baseV}(\X) \times \T)\odotG{\Gamma}(\T\times\E_{\baseV}(\Y))
	     \cong \E_{\baseV}(\X)\odotG{\Gamma}\E_{\baseV}(\Y).
\end{multline*}
Moreover, the assignment  
$\E_{\duplicateV}(\X\uplus\Y)\mapsto  \E_{\baseV}(\X)\odotG{\Gamma}\E_{\baseV}(\Y)$ 
is natural in $\X$ and $\Y$, that is, 
$\E_{\duplicateV}\circ \uplus
\cong 
(\E_{\baseV}\times \E_{\baseV})\circ \odotG{\Gamma}$.

So (up to  natural isomorphism) 
the diagrams in Figure~\ref{Fig:NatDual} commute.
A symbol-chase now confirms  that 
 $\twiddle{\M}$ 
 dualises~$\M$ because~$\twiddle{\N}$ dualises~$\N$:
\begin{multline*}
\hspace*{-.25cm}\E_{\duplicateV}\circ\D_{\duplicateV}
\cong\
\odotG{\Gamma}\circ\ (\E_{\baseV}\times\E_{\baseV})\circ\fnt F\circ\uplus\circ\, (\D_{\baseV}\times\D_{\baseV})\circ\fnt{C}= \\
\hspace*{.3cm}\odotG{\Gamma}\circ\ (\E_{\baseV}\times\E_{\baseV})\circ\, (\D_{\baseV}\times\D_{\baseV})\circ\fnt{C} 
\cong\odotG{\Gamma}\circ\ (\fnt{Id}_{\baseV}\times\fnt{Id}_{\baseV})\circ\fnt{C} 
=\ \odotG{\Gamma}\circ\,\fnt{C} 
\cong \fnt{Id}_{\duplicateV}.
\end{multline*}

 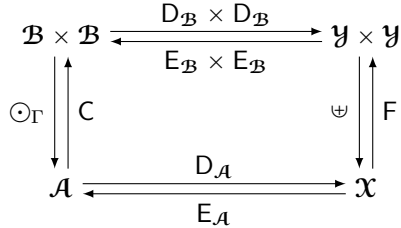
\begin{figure}  [ht]
\begin{center}
\begin{tikzpicture} 
[auto,
 text depth=0.25ex,
 move up/.style=   {transform canvas={yshift=1.9pt}},
 move down/.style= {transform canvas={yshift=-1.9pt}},
 move left/.style= {transform canvas={xshift=-2.5pt}},
 move right/.style={transform canvas={xshift=2.5pt}}] 

\matrix[row sep= 1.5cm, column sep= 1.4cm]
{ 
 \node (B) {$\baseV\times \baseV$}; &&   \node (Y){$\CY\times \CY$}; 
\\ 
 \node (A) {$\duplicateV$};                &&    \node (X) {$\CX$};
\\
};

\draw [-latex , move up] (B) to node  [yshift=-1pt]  {$\D_{\baseV}\times\D_{\baseV}$}(Y);
\draw [latex-, move down] (B) to node [swap] {$\E_{\baseV}\times\E_{\baseV}$}(Y);
\draw [-latex, move left] (B) to node  [swap] {$\odotG{\Gamma}$}(A);
\draw [latex-, move right] (B) to node {$\fnt{C}$} (A);
\draw [-latex , move up] (A) to node  [yshift=-1pt]  {$\D_{\duplicateV}$}(X);
\draw [latex-, move down] (A) to node [swap] {$\E_{\duplicateV}$}(X);
\draw [-latex, move left] (Y) to node  [swap] {$\uplus$}(X);
\draw [latex-, move right] (Y) to node {$\fnt{F}$} (X);
\end{tikzpicture}
\end{center}
\vspace{-15pt}
\caption{Full duality by duplication}\label{Fig:FullDual}
\end{figure}

Assume that $\twiddle{\N}$ yields a full duality. 
 Then   the diagram in Figure~\ref{Fig:FullDual} commutes.
We can easily prove that 
$\D_{\duplicateV}\circ\E_{\duplicateV} \cong \fnt{Id}_{\CX}$,
that is, $\twiddle{\CM}$ yields a full duality.
Moreover, if   $\twiddle{\N}$ is injective in $\CY$ 
then $(\twiddle{\N},\twiddle{\N})$ is injective in $\CY\times\CY$, 
or equivalently 
$\twiddle{\CM}=\twiddle{\N}\uplus\twiddle{\N}$ 
is injective in $\CX$. 
Hence $\twiddle{\CM}$ yields a strong  duality if $\CN$ does.
\end{proof}

Theorem~\ref{thm:transfer-D} applies 
to the variety $p\DBU$ of 
(unbounded) distributive pre-bilattices.
Its members are algebras
$\A = (A; \lor_t,\land_t,\lor_k,\land_k)$ 
 for which 
$(A;\lor_t,\land_t)\in \CCD_u$ and $(A;\lor_k,\land_k)\in \CCD_u$.
The well-known product 
representation for~$p\DBU$ comes  from the observation 
that the set
\[
\Gamma_{p\DBU}=\{
	(\vee^{4}_{13},\wedge^{4}_{24}),
	(\wedge^{4}_{13},\vee^{4}_{24}),
	(\vee^{4}_{13},\vee^{4}_{24}),
	(\wedge^{4}_{13},\wedge^{4}_{24})\}  
\]
satisfies (L), (M) and (D) \cite[Section~9]{prod}.
Since $\twiddle{\two}_u $ strongly dualises  $\CCD_u$, the structure 
 $\twiddle{\two}_u \uplus\twiddle{\two}_u $
determines a  multisorted strong duality for $p\DBU$. 
This was established by different techniques in~\cite[Theorem~10.2]{CPOne}.

Theorem~\ref{thm:transfer-D}  also yields dualities for distributive trilattices.  
 These are (to the best of our knowledge) new. 
As with pre-bilattices, we opt for the unbounded case. 
An unbounded distributive \emph{trilattice} is an algebra 
$(A;\vee_t ,\wedge_t ,\vee_f ,\wedge_f ,\vee_i ,\wedge_i)$
such that 
$(A;\vee_t ,\wedge_t)$, $(A;\vee_f ,\wedge_f)$ and 
$( A; \vee_i,\wedge_i)$ are distributive lattices. 
Let $\TL_u$ denote the variety of (unbounded) distributive trilattices.
An algebra $(A;\vee_t ,\wedge_t ,\vee_f ,\wedge_f ,\vee_i ,\wedge_i,-_{t})$ is  
a \emph{distributive trilattice with $t$-involution} if 
$(A;\vee_t ,\wedge_t ,\vee_f ,\wedge_f ,\vee_i ,\wedge_i)\in\TL_u$ and $-_{t}$ 
is an involution that preserves the $f$- and $i$-lattice operations and reverses $\vee_t$ and $\wedge_t$.
Let $\TL_{-_{t}}$ denote the variety of unbounded distributive trilattices with $t$-involution. 
Take $\DBU$ as the base variety and let 
\begin{multline*}
	\Gammavari{\TL_{-_{t}}}=\{  
		((\land_t)^4_{13},(\land_t)^4_{24}),
		((\lor_t)^4_{13},(\lor_t)^4_{24}),
		((\lor_k)^4_{13},(\land_k)^4_{24}),\\
		((\land_k)^4_{13},(\lor_k)^4_{24}),
		((\land_k)^4_{13},(\land_k)^4_{24}),
		((\lor_k)^4_{13},(\lor_k)^4_{24}),(\neg^{2}_{1},\neg^{2}_2) 
	\}.
\end{multline*}
Then $\Gammavari{\TL_{-_{t}}}$ satisfies (L), (M) and (D) over $\DBU$ 
(see \cite[Example~9.4]{CPOne}).  
In Section~\ref{NatViaDupl}, we used Theorem~\ref{thm:TransferNaturalDualities} 
to prove that $(\twiddle{\two}_u)^2$ yields a strong duality on $\DBU$. 
Now Theorem~\ref{thm:transfer-D} implies that
  $(\twiddle{\two}_u)^2\uplus (\twiddle{\two}_u)^2 $ 
determines a multisorted strong duality for unbounded 
distributive trilattices with  $t$-involution.

We can easily  adapt our results to cater for a base variety which  
admits a multisorted  duality rather than a single-sorted one. 
Predictably this leads to multisortedness at the duplicate level.
In the case of Theorem~\ref{thm:TransferNaturalDualities},
one obtains the  required alter ego by squaring the base level alter ego, 
sort by sort; as before, 
the base variety and its duplicate have the same dual category.
The extension of Theorem~\ref{thm:transfer-D} 
employs two disjoint copies of each sort of the base-level alter ego. 
The proofs of these results involve only
minor modifications of those for the single-sorted case.
As an example, the multisorted version of Theorem~\ref{thm:transfer-D} 
combined with the results in \cite[Example~9.4]{prod} 
leads to a strong duality for unbounded distributive trilattices
which has four sorts, obtained from the 
two-sorted duality for $p\DBU$.

\end{document}